\UseRawInputEncoding
\documentclass{amsart}

\newtheorem{theorem}{Theorem}[section]
\newtheorem{lemma}[theorem]{Lemma}
\newtheorem{proposition}[theorem]{Proposition}

\theoremstyle{definition}
\newtheorem{definition}[theorem]{Definition}
\newtheorem{example}[theorem]{Example}

\theoremstyle{remark}
\newtheorem{remark}[theorem]{Remark}

\def\lu{\underline{\mu}}
\def\ou{\overline{\mu}}
\def\ls{\underline{\sigma}}
\def\os{\overline{\sigma}}
\def\lt{\left}
\def\rt{\right}

\def\vp{\varphi}
\def\ve{\varepsilon}

\def\br{\mathbb{R}}

\def\cp{\mathcal{P}}
\def\cf{\mathcal{F}}

\def\cf{\mathcal{F}}

\def\F{\mathcal{F}}

\def\lu{\underline{\mu}}
\def\ou{\overline{\mu}}
\def\ls{\underline{\sigma}}
\def\os{\overline{\sigma}}

\def\be{\mathbb{E}^{\cp}}

\def\EE{\mathbb{E}^{\cp}}
\def\EEP{\mathbb{E}^{\bar{\cp}}}

\def\vp{\varphi}
\def\ve{\varepsilon}

\numberwithin{equation}{section}

\begin{document}

\title[Pseudo-independence, independence and related limit theorems]{Pseudo-independence, independence and related limit theorems under sublinear expectations}


\author{Li Xinpeng}
\address{Research Center for Mathematics and Interdisciplinary Sciences, Shandong University, 266237, Qingdao, China;
School of Mathematics, Shandong University, 250100, Jinan, China.}
\curraddr{}
\email{lixinpeng@sdu.edu.cn}
\thanks{This work was supported by National Key R\&D Program of China (No.2018YFA0703900), NSF of China (No.11601281) and the Young Scholars Program of Shandong University. The author gratefully acknowledges the many helpful suggestions of Prof. Shige PENG during the preparation of the paper.}


\subjclass[2020]{Primary 60A05; Secondary 60F05}

\date{}

\dedicatory{}


\begin{abstract}
This paper introduces the notion of pseudo-independence on the sublinear expectation space $(\Omega,\mathcal{F},\cp)$ via the classical conditional expectation, and the relations between pseudo-independence and Peng's independence are detailed discussed. Law of large numbers and central limit theorem for pseudo-independent random variables are obtained, and a purely probabilistic proof of Peng's law of large numbers is also given. In the end, some relevant counterexamples are indicated that Peng's law of large numbers and central limit theorem are invalid only with the first and second moment conditions respectively.
\end{abstract}

\maketitle

\section{Introduction}
The notion of independence is very important in the classical probability theory, especially in the limit theorems, which was builded on the probability space $(\Omega,\mathcal{F},P)$, where the classical probability $P$ is additive. However, in some
areas this apparently natural property has been abandoned. Non-additive probabilities or nonlinear expectations as the useful tools for studying uncertainties have received more and more attention in many fields, such as mathematical economics, statistics, quantum mechanics and finance. One typical example of nonlinear expectations is the sublinear expectation, Peng \cite{pengbook} systematically studied the sublinear expectation theory and related stochastic analysis. In particular, the new notions of independence and identical distribution (i.i.d. for short) are introduced in Peng \cite{pengbook} and the related law of large numbers (LLN for short) and central limit theorem (CLT for short) under sublinear expectation are obtained. These notions and theorems are builded on the sublinear expectation space $(\Omega,\mathcal{H},\mathbb{E})$, where $\mathcal{H}$ is a linear space of random variables satisfying certain ``closure property'' for a class of continuous functions and $\mathbb{E}$ is a sublinear functional on $\mathcal{H}$.

In most situations, the sublinear expectation can be represented as the supremum of a set $\cp$ of linear expectations, i.e.,
$\EE[X]=\sup_{P\in\cp}E_P[X]$, see Denis et al. \cite{Peng2011} or Peng \cite{pengbook}. The sublinear expectation space of this paper is slightly different from the original one in Peng \cite{pengbook}. We consider $(\Omega,\mathcal{F},\cp)$ as the sublinear expectation space, where $\cp$ is a set of probability measures on the measurable space $(\Omega,\mathcal{F})$. Following Peng's idea for the notion of independence under sublinear expectation space $(\Omega,\mathcal{H},\mathbb{E})$, we reconsider this notion on sublinear expectation space $(\Omega,\mathcal{F},\cp)$, in particular, the set $\cp$ can be constructed such that Peng's independence also holds under the sublinear expectation $\EE$ introduced by $\cp$. We first introduce the notion of pseudo-independence under sublinear expectation $\EE$ via conditional expectations. Then we can enlarge $\cp$ to $\bar{\cp}$ such that pseudo-independent sequence becomes the independent one under $\bar{\cp}$. Thanks to the existing results of LLN and CLT, we can easily obtain the corresponding LLN and CLT for pseudo-independent sequence. In particular, a purely probabilistic proof of Peng's LLN for i.i.d. sequence is given. But unlike the classical limit theorems, LLN and CLT for i.i.d. sequence may fail only with the finite first and second moment conditions respectively.

The remainder of this paper is organized as follows. Section 2 describes some basic concepts and notations of the sublinear expectation theory. In Section 3, the properties of pseudo-independent sequence are discussed. LLN and CLT for pseudo-independent random variables are proved in Section 4 and Section 5.  The relevant counterexamples of LLN and CLT are provided in Section 6.

\section{Preliminaries}

Given a complete separable metric space $\Omega$, let $\mathcal{F}$ be the Borel $\sigma$-algebra of $\Omega$ and $\mathcal{M}$ be the collection of all probability measures on $(\Omega, \mathcal{F})$. For a subset $\mathcal{P} \subset \mathcal{M}$, the upper expectation of $ \mathcal{P} $ is defined as follows: for each random variable $ X $,
$$
\EE [X]:=\sup _{P \in \mathcal{P}} E_{P}[X].
$$
It is easy to verify that $\EE$ is a sublinear expectation, i.e., $\EE$ satisfies following properties:
\begin{itemize}
\item[(\romannumeral1)] Monotonicity: $\EE[X] \leq \EE[Y],$ if $X \leq Y$,
\item[(\romannumeral2)] Constant preserving: $\EE[c]=c, \forall c \in \mathbb{R}$,
\item[(\romannumeral3)] Sub-additivity: $\EE[X+Y] \leq \EE[X]+\EE[Y]$,
\item[(\romannumeral4)] Positive homogeneity: $\EE[\lambda X]=\lambda \EE[X], \forall \lambda \geq 0$.
\end{itemize}

For each $A\in\mathcal{F}$, the upper and lower probabilities are defined corresponding to the set of probability measures $\mathcal{P}$ respectively as follows:
$$V(A):= \sup _{P \in \mathcal{P}} P(A)$$
$$v(A):= \inf _{P \in \mathcal{P}} P(A).$$
Obviously, $V(\cdot)$ and $v(\cdot)$ are conjugate to each other, that is
$$
V(A)+v\left(A^{c}\right)=1.
$$

The following notions of identical distribution and independence on the sublinear expectation space are initiated by Peng \cite{Peng2007b} (see also Peng \cite{pengsur,pengbook}).

\begin{definition}\label{d1}
Let $X$ and $Y$ be two random variables on sublinear expectation spaces $(\Omega_1, \mathcal{F}_1, {\cp_1})$ and $(\Omega_2, \mathcal{F}_2, {\cp_2})$ respectively. $X$ and $Y$ are called identically distributed, if for each $\varphi \in C_{b, {Lip}}(\mathbb{R})$,
$$
\mathbb{E}^{\cp_1}[\varphi(X)]=\mathbb{E}^{\cp_2}[\varphi(Y)],
$$
where $C_{b,Lip}(\mathbb{R})$ is the set of all bounded and Lipschitz functions on $\br$.
\end{definition}

It is worth pointing out that if $X$ and $Y$ are identical distributed by Definition \ref{d1}, we can not obtain that $V_1(X\in A)=V_2(Y\in A), \forall A\in\mathcal{B}(\br)$ in general, where $V_i$ is introduced by $\cp_i, i=1,2.$

\begin{example}\label{ex2}
Let $(\Omega,\mathcal{F})=([0,1],\mathcal{B}[0,1])$, $\cp_1=\{\delta_x, x\in[0,1)\}$, $\cp_2=\{\delta_x, x\in[0,1]\}$, where $\delta_x$ is Dirac measure at $x$. Consider the random variables $X$ and $Y$ defined by $X(\omega)=Y(\omega)=\omega$. Then for each continuous function $\vp\in C_{b,Lip}(\br)$,
$$\mathbb{E}^{\cp_1}[\vp(X)]=\max_{x\in[0,1]}\vp(x)=\mathbb{E}^{\cp_2}[\vp(Y)],$$
but we have
$$V_1(X=1)=0\neq 1=V_2(X=1).$$
\end{example}
In particular, if we further assume that both  $\cp_1$ and $\cp_2$ are weakly compact, we can deduce that $V_1(X\in A)=V_2(Y\in A), \ \forall A\in\mathcal{B}(\br)$, see Proposition 14 in Hu and Peng \cite{HP} or Proposition 2.8 in Li \cite{li2013}.

\begin{definition}
\label{PI}
Let $\left\{X_{n}\right\}_{n=1}^{\infty}$ be a sequence of random variables on $(\Omega, \mathcal{F}, \cp)$.  $X_{n}$ is said to be independent of $\left(X_{1}, \cdots, X_{n-1}\right)$ under $\cp$, if for each $\varphi \in C_{b, \operatorname{Lip}}\left(\mathbb{R}^{n}\right)$
$$
\EE\left[\varphi\left(X_{1}, \cdots, X_{n}\right)\right]=\EE\left[\left.\EE\left[\varphi\left(x_{1}, \cdots, x_{n-1}, X_{n}\right)\right]\right|_{\left(x_{1}, \cdots, x_{n-1}\right)=\left(X_{1}, \cdots, X_{n-1}\right)}\right].
$$
The sequence of random variables $\left\{X_{n}\right\}_{n=1}^\infty$ is said to be independent, if $X_{n+1}$ is independent of $\left(X_{1}, \cdots, X_{n}\right)$ for each $n \in \mathbb{N}$.
\end{definition}
In the sublinear expectation space, there  also exists $G$-normal distribution which plays an important role as an analogue of the normal distribution in the classical probability theory.
\begin{definition}
A random variable $\xi$ on a sublinear expectation space
$(\Omega,\mathcal{F},\cp)$ is called
$G$-normal distributed, denoted by
$\xi\sim\mathcal{N}(0;[\ls^2,\os^2])$, for a given pair
$0\leq\ls\leq\os<\infty$, if for each $\vp\in C_{b,Lip}(\br)$, the following
function defined by
$$u(t,x)=\EE[\vp(x+\sqrt{t}\xi)],(t,x)\in[0,\infty)\times\br,$$
is the unique viscosity solution of the following parabolic partial
differential equation defined on $[0,\infty)\times\br$:
$$\partial_tu-G(\partial^2_{xx}u)=0, u|_{t=0}=\vp,$$
where
$$G(\alpha)=\frac{1}{2}(\os^2\alpha^+-\ls^2\alpha^-),\alpha\in\br.$$
\end{definition}
In particular, if $\os=\ls$, then $\xi$ defined above is the classical normal distribution.

\section{Properties of pseudo-independent random variables}

Recently, Guo and Li \cite{GL} (see also Li \cite{li2013}) introduced the concept of pseudo-independence, which characterizes the relations between classical conditional expectation and sublinear expectation.
\begin{definition}
Let $\left\{X_{n}\right\}_{n=1}^{\infty}$ be a sequence of random variables on $(\Omega, \mathcal{F}, \cp)$. Define $$\mathcal{F}_{n}:=\sigma\left(X_{1}, \cdots, X_{n}\right), n\geq 1 \quad \text{and} \quad \mathcal{F}_{0}:=\{\emptyset, \Omega\}.$$ If for each $P \in \mathcal{P}$, we have
\begin{equation}
\label{p-ind}
E_{P}\left[ \varphi \lt(X_{n}\rt) \mid \mathcal{F}_{n-1}\right] \leq \EE\left[\varphi \lt(X_{n}\rt)\right] \quad P\text{-a.s.},\quad \forall \varphi \in C_{b, \operatorname{Lip}}(\mathbb{R}),
\end{equation}
then we call that $X_{n}$ is pseudo-independent of $\left(X_{1}, \cdots, X_{n-1}\right)$.
\end{definition}

The pseudo-independent sequence $\{X_n\}_{n=1}^\infty$ can be defined in the same manner as the independent sequence.

\begin{remark}
It is worth pointing out that (\ref{p-ind}) is equivalent to
$$-\EE\left[-\varphi \lt(X_{n}\rt)\right]\leq E_{P}\left[ \varphi \lt(X_{n}\rt) \mid \mathcal{F}_{n-1}\right] \leq \EE\left[\varphi \lt(X_{n}\rt)\right],  \ \forall \varphi \in C_{b, \operatorname{Lip}}(\mathbb{R}).$$ In particular, if $\mathcal{P}=\{P\}$ is a singleton set, as shown in the classical probability theory, $X$ being independent of the $\sigma$-field $\mathcal{F}$ is equivalent to $E_P[\vp(X)|\mathcal{F}]=E_P[\vp(X)]$ for each test function $\vp\in C_{b,Lip}(\mathbb{R})$. However, pseudo-independence means that the conditional expectation $E[\vp(X)|\mathcal{F}]$ may not exactly equal to the unconditional expectation $E_P[\vp(X)]$, but lie in a non-random interval depending on $\vp$.
\end{remark}

Here is an important property of these concepts, which has been proved in Guo and Li \cite{GL}.
\begin{proposition}
\label{conditional expt ineq}
An independent sequence of random variables $\left\{X_{n}\right\}_{n=1}^{\infty}$ on the sublinear expectation space $(\Omega, \mathcal{F}, \cp)$ is pseudo-independent.
\end{proposition}

Pseudo-independence presents more extensive application scenarios than the independence in sublinear expectation space proposed earlier. For instance, independence under each probability $P$ of $\cp$ apparently implies pseudo-independence, but usually can not guarantee independence in Definition \ref{PI}, see Example \ref{ex} in the end of this section. In general, we have the following properties.
\begin{proposition}
\label{prop5}
Let $\{X_n\}_{n=1}^\infty$ be a pseudo-independent sequence under $\cp$, then for each $n\in\mathbb{N}$, $\forall \vp\in C_{b,Lip}(\mathbb{R}^{n+1})$,
$$\EE[\vp(X_1, \cdots, X_n,X_{n+1})]\leq \EE[\EE[\vp(x_1,\cdots,x_n,X_{n+1})]|_{(x_1,\cdots,x_n)=(X_1,\cdots,X_n)}].$$
\end{proposition}
\begin{proof} It is clear that
\begin{align*}
\EE[\vp(X_1,\cdots, X_n,X_{n+1})]=&\sup_{P\in\cp}E_P[\vp(X_1,\cdots,X_{n+1})]\\
=&\sup_{P\in\cp}E_P[E_P[\vp(x_1,\cdots,x_n,X_{n+1})|\mathcal{F}_n]|_{(x_1,\cdots,x_n)=(X_1,\cdots,X_n)}]\\
\leq&\sup_{P\in\cp}E_P[\EE[\vp(x_1,\cdots,x_n,X_{n+1})]_{(x_1,\cdots,x_n)=(X_1,\cdots,X_n)}]\\
=&\EE[\EE[\vp(x_1,\cdots,x_n,X_{n+1})]|_{(x_1,\cdots,x_n)=(X_1,\cdots,X_n)}].
\end{align*}
\end{proof}

Proposition \ref{prop5} inspires us to enlarge $\cp$ to $\bar{\cp}$ such that $\{X_n\}_{n=1}^\infty$ is independent under $\bar{\cp}$.

\begin{proposition}
\label{prop4}
Let $\{X_n\}_{n=1}^\infty$ be a pseudo-independent sequence under $\EE$. We consider $\bar{\cp}$ defined by
$$\bar{\cp}=\{P\in\mathcal{M}:\forall \vp\in C_{b,Lip}(\mathbb{R}), E_P[\vp(X_{n})|\mathcal{F}_{n-1}]\leq\EE[\vp(X_{n})], \forall n\in\mathbb{N}\}$$

then we have $\bar{\cp}\supset\cp$ and $\{X_n\}_{n=1}^\infty$ is independent under $\EEP$ with $\EEP[\vp(X_n)]=\EE[\vp(X_n)], \forall n\in\mathbb{N}, \forall \vp\in C_{b,Lip}(\mathbb{R})$.

Moreover, $\forall n\in\mathbb{N}$,
$$\EEP[\vp(X_1,\cdots,X_{n+1})]=\EE[\EE[\vp(x_1,\cdots,x_n,X_{n+1})]|_{(x_1,\cdots,x_n)=(X_1,\cdots,X_n)}].$$
\end{proposition}
\begin{proof}
For simplicity, we only consider the case of $n=2$. We consider
$$\bar{\cp}:=\{P:\forall \vp\in C_{b,Lip}(\mathbb{R}), E_P[\vp(X_1)]\leq\EE[\vp(X_1)] \ \text{and}\ E_P[\vp(X_{2})|X_1]\leq\EE[\vp(X_{2})]\}.$$
In this case, for each fixed $\vp\in C_{b,Lip}(\mathbb{R}^2)$ and $\ve>0$, there exits $P^*\in\bar{\cp}$ (depending on $\vp$ and $\ve$) such that

(i) $E_{P^*}[\psi(X_1)]\geq\EE[\psi(X_1)]-\frac{\ve}{2}$, where $\psi(x)=\EE[\vp(x,X_2)]$.

(ii) $E_{P^*}[\vp(x,X_2)|X_1]\geq\EE[\vp(x,X_2)]-\frac{\ve}{2}$, for all $x$ belongs to the range of $X_1$.

Then we have
\begin{align*}
\EEP[\vp(X_1,X_2)]\geq& E_{P^*}[\vp(X_1,X_2)]\\
=&E_{P^{*}}[E_{P^{*}}[\vp(x,X_2)|X_1]_{x=X_1}]\\
\geq&E_{P^{*}}[\EE[\vp(x,X_2)]|_{x=X_1}]-\frac{\ve}{2}\\
\geq&\EE[\EE[\vp(x,X_2)]|_{x=X_1}]-\ve\\
=&\EEP[\EEP[\vp(x,X_2)]|_{x=X_1}]-\ve.
\end{align*}
Combining with Proposition \ref{prop5} and noticing that $\ve$ can be arbitrarily small, we see that $X_2$ is independent of $X_1$ under $\bar{\cp}$.
\end{proof}

We conclude this section by an example.

\begin{example}\label{ex}
We consider a probability set $\cp=\{P_1, P_2\}$ and two binary random variables $X$ and $Y$ satisfying
\begin{equation*}
\begin{aligned}
&P_1(X=0, Y=0)=\frac{1}{16},\ P_1(X=1, Y=1)=\frac{9}{16}\\
&P_1(X=1,Y=0)=P_1(X=0, Y=1)=\frac{3}{16}\\
&P_2(X=0, Y=0)=P_2(X=1, Y=1)=P_2(X=0, Y=1)=P_2(X=1,Y=0)=\frac{1}{4}
\end{aligned}
\end{equation*}
We can see that $X$ and $Y$ are independent of each other under  $P_1$ and $P_2$ respectively, it is clear that $Y$ is pseudo-independent of $X$ under $\cp$. Let $\varphi^*$ be a bonded and Lipschitz function with
\begin{equation}\label{vp}
\varphi^*(0,0)=\varphi^*(1,1)=1, \quad \varphi^*(1,0)=\varphi^*(0,1)=0.
\end{equation}
By the simple calculation, we have
\begin{equation*}
\EE\lt[\varphi^*(X,Y)\rt]=\frac{5}{8}<\frac{11}{16}=\EE\lt[\EE\lt[\varphi^*(x,Y)\rt]|_{x=X}\rt],
\end{equation*}
which indicates that $Y$ is not independent of $X$.

Now we enlarge $\cp$ to $\bar{\cp}$ by
$$\bar{\cp}:=\{P: \forall \vp\in C_{b,Lip}(\mathbb{R}), E_P[\vp(X)]\leq\EE[\vp(X)]\ \text{and}\ E_P[\vp(Y)|X]\leq\EE[\vp(Y)]\}.$$
In particular, for $\vp^*$ defined in (\ref{vp}), we can choose $P^*$ by
\begin{align*}
P^*(X=1,Y=1)=\frac{9}{16}\ \ \ &P^*(X=1,Y=0)=\frac{3}{16}\\
P^*(X=0,Y=1)=\frac{1}{8}\ \ \ &P^*(X=0,Y=0)=\frac{1}{8}
\end{align*}
It is easily to verify that $P^*\notin\cp$ but $P^*\in\bar{\cp}$, Furthermore, we have
\begin{align*}
\EEP[\EEP\lt[\varphi^*(x,Y)]|_{x=X}\rt]&=E_{P^*}\lt[E_{P^*}\lt[\varphi^*(x,Y)\rt]|_{x=X}\rt]\\
&=\frac{11}{16}=E_{P^*}[\vp^*(X,Y)]=\EEP[\vp^*(X,Y)].
\end{align*}
\end{example}

\section{Law of large numbers for pseudo-independent random variables}

We firstly consider the law of large numbers for pseudo-independent sequence.

Let $\{X_{n}\}_{n=1}^\infty$ be a sequence of random variables on the sublinear expectation space $(\Omega, \mathcal{F}, \cp),$ and denote $S_{n}=X_{1}+\cdots+X_{n}$, $S_{0}=0$,  $\F_{n}=\sigma(X_{1},X_{2},...,X_{n})$ for $n \geq 1$ and $\F_{0}=\{\emptyset,\Omega\}$. In this section, we require the sequence $\{X_{n}\}_{n=1}^\infty$ consistent with the following hypothesis:
\begin{itemize}
\item[\textit{(H1)}] \textit{There exist a random variable $X$ with $\lim_{n\to\infty} nV(|X|\geq n)=0$ and a constant $K$ such that $V\left(|X_{n}| \geq x\right) \leq K V\lt(|X| \geq x\rt)$ for each $x \geq 0$ and $n \geq 1 .$}
\end{itemize}

We have the following law of large numbers for pseudo-independent sequence.

\begin{theorem}\label{lln}
Let $\{X_i\}_{i=1}^\infty$ be a pseudo-independent sequence under $\cp$ satisfying (H1), we define
$$\ou:=\limsup_{n\rightarrow\infty}\frac{\sum_{i=1}^n\EE[X_i\mathbf{I}_{\{|X_i|< n\}}]}{n} \ \ \ \lu:=\limsup_{n\rightarrow\infty}\frac{\sum_{i=1}^n-\EE[-X_i\mathbf{I}_{\{|X_i|< n\}}]}{n}$$
then we have $\forall \vp\in C_{b,Lip}(\mathbb{R})$,
\begin{equation}\label{eq10}
\min_{\lu\leq\mu\leq\ou}\vp(\mu)\leq\liminf_{n\rightarrow\infty}\EE[\vp(\frac{S_n}{n})]\leq\limsup_{n\rightarrow\infty}\EE[\vp(\frac{S_n}{n})]\leq\max_{\lu\leq\mu\leq\ou}\vp(\mu).
\end{equation}
\end{theorem}

To complete the proof, we need some lemmas. The first lemma is Lemma 3.1 in Guo and Li \cite{GL}.
\begin{lemma} \label{l2} Let $\{X_i\}_{i=1}^\infty$ satisfy (H1), then
\begin{equation}\label{eq2}
\lim_{n\rightarrow\infty}\frac{1}{n^2}\sum_{i=1}^n\be[|X_i|^2\mathbf{I}_{\{|X_i|\leq n\}}]=0.
\end{equation}
\end{lemma}

The next lemma generalizes the corresponding result for i.i.d. sequence in Chen et al. \cite{chen} to the pseudo-independent sequence. Here we give a very simple proof.
\begin{lemma}\label{th1}
Under the same condition of Theorem \ref{lln}, we have, for any $\ve>0$,
$$\lim_{n\to\infty}v(\lu-\ve \leq \frac{S_n}{n} \leq \ou+\ve)=1.$$
\end{lemma}
\begin{proof}
Let $Y_{n,i}=X_i\mathbf{I}_{\{|X_i|< n\}}$ and $T_n=\sum_{i=1}^n Y_{n,i}$.  For each $P\in\cp$, obviously, $E_P[Y_{n,i}^2]\leq\be[X_i^2]$, and $-\EE[-Y_{n,i}]\leq E_P[Y_{n,i}|\cf_{i-1}]\leq\EE[Y_{n,i}]$.

Firstly, by Chebyshev's inequality and Jensen's inequality, we have
\begin{align*}P(\frac{\sum_{i=1}^n(Y_{n,i}-E_P[Y_{n,i}|\cf_{i-1}])}{n}>\ve)&\leq\frac{\sum_{i=1}^n(E_P[Y_{n,i}^2]+E_P[(E_P[Y_{n,i}|\cf_{i-1}])^2])}{n^2\ve^2}\\
&\leq\frac{2\sum_{i=1}^nE_P[Y_{n,i}^2]}{n^2\ve^2}\leq \frac{2\sum_{i=1}^n\be[Y_{n,i}^2]}{n^2\ve^2}.
\end{align*}
Noting that,
$$P(S_n\neq T_n)\leq\sum_{i=1}^nP(|X_i|\geq n)\leq \sum_{i=1}^nV(|X_i|\geq n)\leq nKV(|X|\geq n),$$
thus
\begin{align*}
P(\frac{S_n}{n}>\ou+\ve)&\leq P(S_n\neq T_n)+P(\frac{\sum_{i=1}^n(Y_{n,i}-E[Y_{n,i}|\cf_{i-1}])}{n}>\frac{\ve}{2})\\
&\ \ \ \ +P(\frac{\sum_{i=1}^nE[Y_{n,i}|\cf_{i-1}]}{n}>\ou+\frac{\ve}{2})\\
&\leq nKV(|X|\geq n)+\frac{8\sum_{i=1}^n\be[Y_{n,i}^2]}{n^2\ve^2}+P(\frac{\sum_{i=1}^n\be[Y_{n,i}]}{n}>\ou+\frac{\ve}{2}).
\end{align*}
Combining with Lemma \ref{l2}, we can imply that
$$V(\frac{S_n}{n}>\ou+\ve)\rightarrow 0,$$
Similarly, we can prove that $V(\frac{S_n}{n}<\lu-\ve)\rightarrow 0$, the proof is completed.
\end{proof}

Now we give the proof of Theorem \ref{lln}.

\begin{proof}[Proof of Theorem \ref{lln}]
If we take $\vp'=-\vp$ in (\ref{eq10}), we only need to prove that
$$\limsup_{n\rightarrow\infty}\EE[\vp(\frac{S_n}{n})]\leq\max_{\lu\leq\mu\leq\ou}\vp(\mu).$$
For each $\ve>0$, we have
$$\be[\vp(\frac{S_n}{n})]\leq\max_{\lu-\ve\leq\mu\leq\ou+\ve}\vp(\mu)+\max_{\mu\in\br}\vp(\mu)(V(\frac{S_n}{n}>\ou+\ve)+V(\frac{S_n}{n}<\lu-\ve)).$$
Since $\vp$ is bounded, letting $n\rightarrow\infty$, we can complete the proof.
\end{proof}

We reformulate Peng's LLN (see Theorem 2.4.1 in Peng \cite{pengbook}) for the probability set $\cp$ and give a purely probabilistic proof. Such probabilistic proof was firstly considered in Li \cite{li2013} (see Theorem 2.2) under the condition of $\EE[|X_1|^{1+\delta}]<\infty$ for some $\delta>0$. The current proof is totally different from the original one in Peng \cite{pengbook} based on the PDE estimation.

\begin{theorem}\label{PengLLN}
Let $\{X_i\}_{i=1}^\infty$ be an i.i.d. sequence under $\EE$ with $\lim_{\lambda\rightarrow+\infty}\EE[(|X_1|-\lambda)^+]=0$, then
\begin{equation}\label{eq20}
\lim_{n\rightarrow\infty}\mathbb{E}[\vp(\frac{X_1+\cdots+X_n}{n})]=\max_{-\mathbb{E}[-X_1]\leq\mu\leq\mathbb{E}[X_1]}\vp(\mu),
\end{equation}
where $\vp$ is any continuous function satisfying linear growth condition.
\end{theorem}

\begin{proof}
We only need to prove that (\ref{eq20}) holds for $\vp\in C_{b,Lip}(\br)$. By the same argument in Peng \cite{pengbook}, it can be easily generalized to $\vp$ with linear growth condition.

It is clear that $nV(|X_1|\geq n)\rightarrow 0$ and $\EE[\pm X_1\mathbf{I}_{\{|X_1|<n\}}]\rightarrow\EE[\pm X_1]$ since $\lim_{\lambda\rightarrow+\infty}\EE[(|X_1|-\lambda)^+]=0$. Then by Theorem \ref{lln}, we have
$$\limsup_{n\rightarrow\infty}\EE[\vp(\frac{S_n}{n})]\leq\max_{-\EE[-X_1]\leq\mu\leq\EE[X_1]}\vp(\mu).$$
On the other hand, for fixed $\vp\in C_{b,Lip}(\br)$ with Lipschitz constant $L_\vp$ and bound $C_\vp$,  and for each ${\ve}>0$, there exists  $\mu\in(-\EE[-X_1],\EE[X_1])$ such that $\vp(\mu)\geq \max_{-\EE[-X_1]\leq\mu\leq\EE[X_1]}\vp(\mu)-{\ve}$, we can find $E_P$ such that $\{X_i\}$ is an i.i.d. sequence under $E_P$ with $E_P[X_1]=\mu$, where $P$ can be chosen from $\bar{\cp}$ defined in Proposition \ref{prop4}.  In this case, by the classical weak law of large numbers,
$$P(|\frac{S_n}{n}-\mu|>\ve)\rightarrow 0.$$
Then
$$|E_P[\vp(\frac{S_n}{n})]-\vp(\mu)|\leq L_\vp\ve+2C_\vp P(|\frac{S_n}{n}-\mu_n|>\ve),$$
We can obtain that
$$\lim_{n\rightarrow\infty}|E_P[\vp(\frac{S_n}{n})]-\vp(\mu)|=0,$$
which implies that, for each ${\ve}>0$,
\begin{equation}
\liminf_{n\rightarrow\infty}\be[\vp(\frac{S_n}{n})]\geq\max_{-\EE[-X_1]\leq\mu\leq\EE[X_1]}\vp(\mu)-\ve.
\end{equation}
Let $\ve\rightarrow 0$, the LLN holds.
\end{proof}

\section{Central limit theorem for pseudo-independent random variables}

In this section, we consider the CLT for pseudo-independent sequence. We firstly reformulate the Peng's CLT in Peng \cite{pengbook} on sublinear expectation space $(\Omega,\mathcal{F},\cp)$. The further result, for example, the convergence rate of CLT can be found in Fang et al. \cite{fang} and Song \cite{song}.

\begin{theorem} \label{PengCLT}
Let $\{X_i\}_{i=1}^\infty$ be an i.i.d. sequence of random variables on $(\Omega,\mathcal{F},\cp)$, we assume further
$$\lim_{\lambda\rightarrow+\infty}\EE[(|X_1|^2-\lambda)^+]=0,$$
then
$$\lim_{n\rightarrow\infty}\EE[\vp(\frac{X_1+\cdots+X_n}{\sqrt{n}})]=\EE[\vp(\xi)],$$
for all continuous function $\vp$ with linear growth condition, where $\xi$ is a $G$-normally distributed random variable with $\xi\sim\mathcal{N}(0,[-\EE[-X_1^2],\EE[X_1^2]])$.
\end{theorem}

Recently, Zhang \cite{zhang2021} has weaken the condition $\lim_{\lambda\rightarrow+\infty}\mathbb{E}[(|X_1|^2-\lambda)^+]=0$ to $nV(|X_1|^2\geq n)\rightarrow 0$, which inspires us to assume the sequence $\{X_{n}\}_{n=1}^\infty$ consistent with the following hypothesis:
\begin{itemize}
\item[\textit{(H2)}] \textit{There exist a random variable $X$ with $\lim_{n\to\infty} nV(X^2\geq n)=0$ and a constant $K$ such that $V\left(|X_{n}| \geq x\right) \leq K V\lt(|X| \geq x\rt)$ for each $x \geq 0$ and $n \geq 1 .$}
\end{itemize}

Let $S_n=X_1+\cdots+X_n$, we have the following central limit theorem for pseudo-independent sequence.

\begin{theorem}\label{clt}
Let $\{X_i\}_{i=1}^\infty$ be a pseudo-independent sequence under $\cp$ satisfying (H2) and $\EE[X_i]=\EE[-X_i]=0, \forall i\in\mathbb{N}$, we define
$$\os^2:=\limsup_{i\rightarrow\infty}\EE[X_i^2],\ \ \ \ls^2:=\liminf_{i\rightarrow\infty}-\EE[-X_i^2].$$
We further assume that $\os<+\infty$, then we have, $\forall \vp\in C_{b,Lip}(\br)$,
$$-\EE[-\vp(\xi)]\leq\liminf_{n\rightarrow\infty}\EE[\vp(\frac{S_n}{\sqrt{n}})]\leq\limsup_{n\rightarrow\infty}\EE[\vp(\frac{S_n}{\sqrt{n}})]\leq\EE[\vp(\xi)], $$
where $\xi$ is a $G$-normally distributed random variable with $\xi\sim\mathcal{N}(0,[\ls^2,\os^2])$.
\end{theorem}

To prove Theorem \ref{clt}, we need the following lemmas.
\begin{lemma}
\label{lm1}
Let $\{X_{ij}\}_{1\leq j\leq i}^\infty$ be the double array,
where for each row, $\{X_{ni}\}_{i=1}^n$ is independent sequence under $\EE$.
We assume that
$$\frac{1}{\sqrt{n}}\sum_{i=1}^n(\EE[|X_{ni}|]+\EE[-X_{ni}])\rightarrow 0,$$
and
$$\frac{1}{n^{3/2}}\sum_{i=1}^n\EE[|X_{ni}|^3]\rightarrow 0.$$
Define $\os^*=\sup_{i,j}\EE[X_{ij}^2]$ and $\ls^*=\inf_{i,j}-\EE[-X_{ij}^2]$, we further assume that $\os<+\infty$.
Then we have
$$-\EE[-\vp(\xi^*)]\leq\liminf_{n\rightarrow\infty}\EE[\vp(\frac{S_n}{\sqrt{n}})]\leq\limsup_{n\rightarrow\infty}\EE[\vp(\frac{S_n}{\sqrt{n}})]\leq\EE[\vp(\xi^*)],$$
where $\xi$ is a $G$-normally distributed random variable with $\xi^*\sim\mathcal{N}(0,[{\ls^*}^2,{\os^*}^2])$.
\end{lemma}

The proof is very similar to the corresponding one in Li \cite{li}, so we omit the proof here.

\begin{lemma}
\label{lm2}
Let $\{X_i\}_{i=1}^\infty$ be an independent sequence under $\EE$ satisfying (H2) and $\EE[X_i]=\EE[-X_i]=0, \forall i\in\mathbb{N}$. Define $\os^*=\sup_{i}\EE[X_{i}^2]$ and $\ls^*=\inf_{i}-\EE[-X_{i}^2]$, we further assume that $\os<+\infty$.
Then we have
$$-\EE[-\vp(\xi^*)]\leq\liminf_{n\rightarrow\infty}\EE[\vp(\frac{S_n}{\sqrt{n}})]\leq\limsup_{n\rightarrow\infty}\EE[\vp(\frac{S_n}{\sqrt{n}})]\leq\EE[\vp(\xi^*)],$$
where $\xi$ is a $G$-normally distributed random variable with $\xi^*\sim\mathcal{N}(0,[{\ls^*}^2,{\os^*}^2])$.
\end{lemma}

\begin{proof}
Let $X_{ni}=(-\sqrt{n})\vee X_i\wedge\sqrt{n}$, by the same argument in Zhang \cite{zhang2021}, $\{X_{ni}\}$ satisfies the conditions in Lemma \ref{lm1}, we have
$$\limsup_{n\rightarrow\infty}\EE[\vp(\frac{\sum_{i=1}^n X_{ni}}{\sqrt{n}})]\leq\EE[\vp(\xi^*)].$$
It is clear that
$$\EE[|\vp(\frac{\sum_{i=1}^n X_{ni}}{\sqrt{n}})-\vp(\frac{S_n}{\sqrt{n}})|]\leq\sup_{x}|\vp(x)|\sum_{i=1}^nV(|X_i|\geq\sqrt{n})\leq KnV(|X|\geq\sqrt{n})\rightarrow 0,$$
we can complete the proof.
\end{proof}

Now we give the proof of Theorem \ref{clt}.

\begin{proof}[Proof of Theorem \ref{clt}]
By Proposition \ref{prop4}, we can enlarge $\cp$ to $\bar{\cp}$ such that $\{X_i\}_{i=1}^\infty$ is independent under $\EEP$ with $\EEP[\pm X_i^p]=\EE[\pm X_i^p]$, $p=1,2$. We further assume that $\ls>0$. In fact, this condition can be removed by the perturbation method in
the proof of Theorem 2.4.7 in Peng \cite{pengbook}.

For each $0<\ve<\ls$, we can choose $N$ such that $\EEP[X_i^2]\leq \os^2+\ve$ and $-\EEP[-X_i^2]\geq \ls^2-\ve$, when $i\geq N$. Set $S_{n}^N=\sum_{i=N}^nX_i$, for $n\geq N$, by Lemma \ref{lm2}, we have
$$\limsup_{n\rightarrow\infty}\EEP[\vp(\frac{S_n^N}{\sqrt{n-N}})]\leq\EEP[\vp(\xi_\ve)],$$
where $\xi_\ve\sim\mathcal{N}(0,[\ls^2-\ve,\os^2+\ve])$.

It is easily seen that
\begin{align*}
|\EEP[\vp(\frac{S_n}{\sqrt{n}})-\vp(\frac{S_n^N}{n-N})]|&\leq\EEP[|\vp(\frac{S_n}{\sqrt{n}})-\vp(\frac{S_n^N}{\sqrt{n}})|]+\EEP[|\vp(\frac{S_n^N}{\sqrt{n}})-\vp(\frac{S_n^N}{\sqrt{n-N}})|]\\
&\leq C\frac{\sum_{i=1}^N\EEP[|X_i|]}{\sqrt{n}}+C\frac{N\sum_{i=N}^n\EEP[|X_i|]}{\sqrt{n(n-N)(\sqrt{n}+\sqrt{n+N})}}\rightarrow 0,
\end{align*}
Thus we obtain
$$\limsup_{n\rightarrow\infty}\EEP[\vp(\frac{S_n}{\sqrt{n}})]\leq\EE[\vp(\xi_\ve)].$$
Note that $\EE[\vp(\frac{S_n}{\sqrt{n}})]\leq\EEP[\vp(\frac{S_n}{\sqrt{n}})]$, and $\EE[\vp(\xi_\ve)]\rightarrow\EE[\vp(\xi)]$ as $\ve\rightarrow 0$, the proof is complete.
\end{proof}

\section{Example}

 In this section, we will present relevant counterexamples to illustrate that the LLN and CLT for i.i.d sequence may fail only with $\EE[|X_1|]<\infty$ and $\EE[X_1^2]<\infty$ respectively.

\begin{example}
\label{exam}
Let $\Omega = \mathbb{Z}$, $\mathcal{F}=\mathcal{B}(\mathbb{Z})$, $\mathcal{P}=\{P_{k}, k \geq 1\}$, where $P_{k}(\{0\})=1-\frac{1}{k^2}$, $P_{k}(\{k\})=P_{k}(\{-k\})=\frac{1}{2k^2}$. Consider a function $X$ on $\mathbb{Z}$ defined by $X(n)=n,n \in \mathbb{Z}$. It is obvious that $\EE[X]=\EE[-X]=0$ and  $\EE[X^{2}]=-\EE[-X^{2}]=1$. According to Proposition \ref{prop4} or Peng \cite{pengbook}, we are able to construct an i.i.d. sequence $\{X_{n}\}_{n=1}^\infty$ under $\EE$ such that $X_{n}$ has the same distribution with $X$.
\end{example}

We have the following properties for such example.

\begin{proposition}
\label{prop3}
In Example \ref{exam}, let $Y_i=X_i^2$ and $\vp(x)=1-x$, then
$$\lim_{n\rightarrow\infty}\EE[\vp(\frac{Y_1+\cdots+Y_n}{n})]=1.$$
\end{proposition}
\begin{proof}
It is clear that,
\begin{align*}
\EE[\vp(\frac{x+Y_i}{n})]&=\sup_{k\in\mathbb{N}}\{(1-\frac{1}{k^2})\vp(\frac{x}{n})+\frac{1}{k^2}(\vp(\frac{x+k^2}{n})\}\\
&=\sup_{k\in\mathbb{N}}\{\vp(\frac{x}{n})+\frac{1}{k^2}(\vp(\frac{x+k^2}{n})-\vp(\frac{x}{n}))\}=\vp(\frac{x}{n}).
\end{align*}
Then
\begin{align*}
\EE[\vp(\frac{Y_1+\cdots+Y_n}{n})]&=\EE[\EE[\vp(\frac{x+Y_n}{n})]|_{x=Y_1+\cdots+Y_{n-1}}]\\
&=\EE[\vp(\frac{Y_1+\cdots+Y_{n-1}}{n})]=\cdots=\EE[\vp(\frac{Y_1}{n})]=\vp(0)=1.
\end{align*}
\end{proof}
Proposition \ref{prop3} indicates that the Peng's LLN (see Theorem \ref{PengLLN}) does not hold only with the finite moment condition. Indeed,
$\lim_{n\rightarrow\infty}\EE[\varphi(\frac{S_n}{\sqrt{n}})]=1\neq 0=\vp(1)$.

\begin{proposition}
\label{prop2}
In Example \ref{exam}, let $\varphi(x)=1-|x|$, then for all $n\in\mathbb{N}$,
\begin{align}\label{eq4}
\EE[\varphi(\frac{X_1+\cdots+X_n}{\sqrt{n}})]=1,
\end{align}
thus the central limit theorem does not hold.
\end{proposition}
\begin{proof}
We first observe that
\begin{align*}\EE[\vp(\frac{x+X}{\sqrt{n}})]&=\sup_{k\in\mathbb{N}}\{(1-\frac{1}{k^2})\vp(\frac{x}{\sqrt{n}})+\frac{1}{2k^2}(\vp(\frac{x-k}{\sqrt{n}})+\vp(\frac{x+k}{\sqrt{n}}))\}\\
&=\vp(\frac{x}{\sqrt{n}})+\frac{1}{\sqrt{n}}\sup_{k\in\mathbb{N}}\{\frac{2|x|-|x+k|-|x-k|}{2k^2}\}=\vp(\frac{x}{\sqrt{n}}).
\end{align*}
Then
\begin{align*}
\EE[\vp(\frac{S_n}{\sqrt{n}})]&=\EE[\EE[\vp(\frac{x+X_n}{\sqrt{n}})]|_{x=X_1+\cdots+X_{n-1}}]=\EE[\vp(\frac{X_1+\cdots+X_{n-1}}{\sqrt{n}})]\\
&=\cdots=\EE[\vp(\frac{X_1}{\sqrt{n}})]=\vp(0)=1.
\end{align*}

In this case,
$\lim_{n\rightarrow\infty}\EE[\varphi(\frac{S_n}{\sqrt{n}})]=1>\mathcal{N}(\vp)$, where $\mathcal{N}(\vp)$ is the expectation of $\vp(\xi)$ and $\xi$ is standard normal distributed. Peng's CLT (Theorem \ref{PengCLT}) is invalid.
\end{proof}

In the end of the paper, we raise the following questions as the open problems:

{\bf Questions} Do the limit distributions of $\frac{X_1+\cdots+X_n}{\sqrt{n}}$ and $\frac{X_1^2+\cdots+X_n^2}{n}$ in Example \ref{exam} (or in general case) exist? If they exist, what are the corresponding limit distributions?

\providecommand{\bysame}{\leavevmode\hbox to3em{\hrulefill}\thinspace}
\providecommand{\MR}{\relax\ifhmode\unskip\space\fi MR }
\providecommand{\MRhref}[2]{%
  \href{http://www.ams.org/mathscinet-getitem?mr=#1}{#2}
}
\providecommand{\href}[2]{#2}

\end{document}